\documentclass{article} 
\author{Andy Hammerlindl}
\title{Dynamics of quasi-isometric foliations}

\usepackage{amsmath}
\usepackage{amssymb}
\usepackage{amsfonts}
\usepackage{amsthm}
\usepackage{graphicx}

\newcommand{\Heis}{\mathcal{H}}
\newcommand{\h}{\mathfrak{h}}
\newcommand{\Lspace}{\mathfrak{L}}
\newcommand{\bbR}{\mathbb{R}}
\newcommand{\R}{\mathbb{R}}
\newcommand{\Rd}{\R^d}
\newcommand{\bbZ}{\mathbb{Z}}
\newcommand{\Z}{\mathbb{Z}}
\newcommand{\bbT}{\mathbb{T}}
\newcommand{\T}{\mathbb{T}}
\newcommand{\bbS}{\mathbb{S}}
\newcommand{\F}{\mathcal{F}}
\newcommand{\Es}{E^s}
\newcommand{\Ec}{E^c}
\newcommand{\Eu}{E^u}
\newcommand{\Ecu}{E^{cu}}
\newcommand{\Ecs}{E^{cs}}
\newcommand{\Ws}{W^s}
\newcommand{\Wc}{W^c}
\newcommand{\Wu}{W^u}
\newcommand{\Wcu}{W^{cu}}
\newcommand{\Wcs}{W^{cs}}
\newcommand{\inv}{^{-1}}
\newcommand{\myrel}{\asymp}

\newcommand{\dist}{\operatorname{dist}}
\newcommand{\length}{\operatorname{length}}
\newcommand{\diam}{\operatorname{diam}}
\newcommand{\id}{\operatorname{id}}

\newcommand{\ttphi}{\tilde{\tilde\phi}}

\newtheorem{thm}{Theorem}[section]
\newtheorem{cor}[thm]{Corollary}
\newtheorem{lemma}[thm]{Lemma}
\newtheorem{prop}[thm]{Proposition}
\newtheorem{question}[thm] {\bf Question}

\theoremstyle{remark}
\newtheorem*{remark} {\bf Remark}
\newtheorem*{defn} {\bf Definition}

\providecommand{\acknowledgement}{{\noindent\bf Acknowledgements}\quad}

\begin{document}

\maketitle

\begin{abstract}
    If the stable, center, and unstable foliations of a partially hyperbolic
    system are quasi-isometric, the system has Global Product Structure.  This
    result also applies to Anosov systems and to other invariant splittings.

    If a partially hyperbolic system on a manifold with an abelian fundamental
    group has quasi-isometric stable and unstable foliations, the center
    foliation is without holonomy.  If, further, the system has Global Product
    Structure, then all center leaves are homeomorphic.
\end{abstract}
\section{Introduction} 

Invariant foliations are invaluable in understanding the properties of smooth
dynamical systems.
In this paper, we examine foliations invariant under a
dynamical system which have the additional property of quasi-isometry.
A foliation $W$ is \emph{quasi-isometric} if there is a constant $Q>1$ such
that for any two points $x,y \in M$ which lie on the same leaf of $W$,
$d_W(x,y)  \le  Q d_M(x,y) + Q$
where $d_W$ is distance measured along the leaf and $d_M$ is distance measured
on the manifold.
Such foliations arise naturally in the study of certain partially hyperbolic
systems.
A diffeomorphism $f:M \to M$ on a compact Riemannian manifold is \emph{partially
hyperbolic} if there are a $Tf$-invariant splitting,
$TM = \Es \oplus \Ec \oplus \Eu$,
an integer $n  \ge  1$, and constants $\lambda < \hat \gamma < 1 < \gamma <
\mu$ such that
\[
    \|Tf^n v^s\| < \lambda < \hat \gamma < \|Tf^n v^c\|
    < \gamma < \mu < \|Tf^n v^u\|
\]
for all $x \in M$, and unit vectors $v^s \in \Es(x)$, $v^c \in E^c(x)$, and
$v^u \in E^u(x)$.

For every partially hyperbolic system, there are unique foliations $\Wu$ and
$\Ws$ tangent to the unstable $\Eu$ and stable $\Es$ subbundles.  In many but not
all cases, there is also a foliation tangent to the center subbundle $\Ec$.
If the three subbundles $\Ec$, $\Ecs=\Ec \oplus \Es$, and $\Ecu = \Ec \oplus \Eu$ are
uniquely integrable, the system is said to be \emph{dynamically coherent}.
M.~Brin proved that if the foliations $\Wu$ and $\Ws$ are quasi-isometric when
lifted to the universal cover of the manifold, the system is dynamically
coherent \cite{Brin}.  Later, M.~Brin, D.~Burago, and S.~Ivanov proved that every
partially hyperbolic system on the 3-torus is dynamically coherent by
establishing this quasi-isometry property \cite{BBI2}.

The property of quasi-isometry tells us much more about partially hyperbolic
systems than just dynamical coherence. In my thesis, I used the property to
give a classification result for partially hyperbolic systems on the 3-torus
\cite{ham-thesis}.

This paper explores the consequences of quasi-isometry for partially
hyperbolic systems on general manifolds and in any dimension $d  \ge  3$.  The
first result is on the structure of the invariant foliations.

\begin{thm} \label{introgps}
    Suppose $f:M \to M$ is partially hyperbolic and the foliations
    $\Wu$, $\Ws$, and $\Wc$ are quasi-isometric when lifted to the universal
    cover, $\tilde M$.  Then, the foliations $\Wcs$ and $\Wcu$ tangent to $\Ec
    \oplus \Es$ and $\Ec \oplus \Eu$ are also quasi-isometric on $\tilde M$, and $f$
    has \emph{Global Product Structure}, that is, for every $x,y \in \tilde M${:}
    \begin{enumerate}
        \item $\Wu(x)$ and $\Wcs(y)$ intersect exactly once,
        \item $\Ws(x)$ and $\Wcu(y)$ intersect exactly once,
        \item if $x \in \Wcs(y)$, then $\Wc(x)$ and $\Ws(y)$ intersect exactly
        once, and
        \item if $x \in \Wcu(y)$, then $\Wc(x)$ and $\Wu(y)$ intersect exactly
        once.
    \end{enumerate}  \end{thm}
In the classification given for systems on the 3-torus, the first key step of
the proof was establishing Global Product Structure. As such, Theorem
\ref{introgps} may help to classify systems on other manifolds and in
higher dimensions.  Section \ref{apps} gives specific examples of its
application and the proof of the theorem is given in Section \ref{finer}.

There are partially hyperbolic systems where $\Wu$ and $\Ws$ are
quasi-isometric, but $\Wc$, $\Wcs$, and $\Wcu$ are not \cite{ham-nil}.  These
examples still have Global Product Structure.

\begin{question} \label{qgps}
    If $f:M \to M$ is partially hyperbolic, and the foliations $\Wu$ and $\Ws$
    (but not necessarily $\Wc$) are quasi-isometric on the universal cover,
    does $f$ have Global Product Structure?
\end{question}
On the 3-torus, partially hyperbolic systems fall into two categories:
skew products, where every center leaf is a compact circle, and
``Derived-from-Anosov'' systems, where the center foliation consists entirely
of lines.
In either case, any two leaves of the center foliation are homeomorphic, and
this is indicative of a more general principle.

\begin{thm} \label{introabel}
    Suppose $M$ is a compact manifold with an abelian fundamental group, $f:M \to
    M$ is partially hyperbolic with Global Product Structure, and
    $\Ws$ and $\Wu$ are quasi-isometric on the universal cover.  Then, any two
    leaves of $\Wc$ are homeomorphic.
\end{thm}
\begin{cor}
    Suppose $M$ is a compact manifold with an abelian fundamental group, $f:M \to
    M$ is partially hyperbolic and $\Wu$, $\Wc$, and $\Ws$ are quasi-isometric on
    the universal cover.  Then, any two leaves of $\Wc$ are homeomorphic.
\end{cor}

The assumption of an abelian fundamental group cannot be removed, as
shown by the following example.  Let $M = \T^2 \times \bbS^2$ and let $A:\T^2
\to \T^2$ be a hyperbolic toral automorphism so that $f = A \times \id:M \to M$
is partially hyperbolic and each center leaf is a sphere $\{x\} \times \bbS^2$.
If $-y$ denotes the antipode of $y \in \bbS^2$, then $f$ commutes with $\tau:M
\to M$, $(x,y) \mapsto (-x,-y)$ and the resulting diffeomorphism on the quotient
space $M/\tau$ has four leaves homeomorphic to the projective plane while the
rest of the leaves are spheres.

\begin{question} \label{qhomeo}
    If $f:M \to M$ is partially hyperbolic and $\Wu$ and $\Ws$ are
    quasi-isometric when lifted to the universal cover, then is there a finite
    cover of $M$ on which any two leaves of $\Wc$ homeomorphic?
\end{question}
The key to proving Theorem \ref{introabel} is to consider holonomies along the
center foliation.  We say that $\Wc$ is \emph{without holonomy} if the holonomy
along every closed loop on a leaf of $\Wc$ is the identity.  Using this,
Theorem \ref{introabel} can be decomposed into two results, each interesting in
its own right.

\begin{thm} \label{intronoholo}
    Suppose $M$ is a compact connected manifold with an abelian fundamental
    group, and $f:M \to M$ is partially hyperbolic.  If $\Wu$ and
    $\Ws$ are quasi-isometric when lifted to the universal cover, then the
    foliations $\Wc$, $\Wcs$, and $\Wcu$ are without holonomy.
\end{thm}
\begin{thm} \label{introhomeo}
    Suppose $M$ is a connected manifold and $f:M \to M$ is partially
    hyperbolic. If the center foliation $\Wc$ exists and is without holonomy, and
    $f$ has Global Product Structure, then all leaves
    of $\Wc$ (on $M$) are homeomorphic.
\end{thm}
The proof of Theorem \ref{intronoholo} is based on properties of liftings of
freely homotopic curves.  These properties rely on the fundamental group being
abelian and do not hold even for the most obvious generalization to manifolds
with nilpotent fundamental groups.  That said, the only counterexamples we
know of are systems such as the example given above, where the center
foliation lifts to a trivial fibering on a finite cover.

\begin{question} \label{qnoholo}
    Suppose $M$ is a compact connected manifold, $f:M \to M$ is 
    partially hyperbolic, and $\Wu$ and $\Ws$ are quasi-isometric when lifted to
    the universal cover.  Modulo a lift to a finite cover, is the center
    foliation $\Wc$ without holonomy?
\end{question}
The author's aim in researching this topic was to determine exactly which
3-dimensional manifolds support partially hyperbolic systems with
quasi-isometric invariant foliations.  Alas, this goal has remained elusive.

\begin{question} \label{qnil}
    For which manifolds $M$ is there $f:M \to M$ is partially hyperbolic, with
    $\dim \Es=\dim \Ec=\dim \Eu=1$, and $\Wu$ and $\Ws$ quasi-isometric on the
    universal cover?
\end{question}
The following fact shows that positive answers to Questions \ref{qgps} and
\ref{qnoholo} will give an answer to Question \ref{qnil}.

\begin{thm} \label{easynil}
    If $f:M \to M$ is partially hyperbolic with Global Product Structure,
    $\dim \Es=\dim \Ec=\dim \Eu=1$,
    and $\Wcs$ (or $\Wcu$) is without holonomy, then $M$
    a circle bundle over $\bbT^2$.
\end{thm}
Section \ref{secholo} contains the proofs of 
Theorems \ref{intronoholo}, \ref{introhomeo}, and \ref{easynil}.

\medskip

Above, we stated Theorem \ref{introgps} in terms of the three-way splitting
of a partially hyperbolic system. In fact, the theorem generalizes to
arbitrary $n$-way splittings.  As this requires more notation to state, we
leave the exposition until Section \ref{finer}.
In an Anosov system, $TM$ has a two-way splitting into contracting $\Es$ and
expanding $\Eu$ bundles, and the result can be stated as follows.

\begin{thm} \label{introanosov}
    If $f:M \to M$ is Anosov and $\Wu$ and $\Ws$ are quasi-isometric on the
    universal cover $\tilde M$, then $f$ has Global Product Structure, that is,
    for every $x,y \in \tilde M$, the intersection $\Ws(x) \cap \Wu(y)$ consists of
    exactly one point.
\end{thm}
If $f:M \to M$ is an Anosov diffeomorphism, it is a famous open question if $M$
is finitely covered by a nilmanifold.  Further, it is not known
if the non-wandering set of $f$ is equal to $M$ or if the universal
cover of $M$ is homeomorphic to $\R^n$.  In \cite{Brin-nw}, Brin defined
pinching conditions on the derivative of an Anosov system under which the
system necessarily has Global Product Structure (referred to in that paper as
the property of ``infinite extension'').  He further showed that if an Anosov
diffeomorphism has Global Product Structure, then its non-wandering set is the
entire manifold, and the universal cover is homeomorphic to $\R^n$.  As such,
Theorem \ref{introanosov} shows that any Anosov diffeomorphism with
quasi-isometric foliations also enjoys these properties.

\begin{cor}
    If $f:M \to M$ is Anosov and $\Wu$ and $\Ws$ are quasi-isometric on the
    universal cover $\tilde M$, then $\Omega(f)=M$ and $\tilde M$ is homeomorphic
    to $\bbR^n$.
\end{cor}
One can find hyperbolic toral automorphisms which are not pinched, but as the
invariant foliations are linear, they are quasi-isometric when lifted to the
universal cover.  The work of J.~Franks and A.~Manning shows that every Anosov
diffeomorphism on a torus must have quasi-isometric lifted foliations
\cite{Franks1}\cite{Manning}.  In Section \ref{anosovexamples}, we examine a
six-dimensional nilmanifold supporting Anosov diffeomorphisms, some of which
have quasi-isometric lifted foliations while others do not.

\section{Applications} \label{apps} 

Attempts to classify partially hyperbolic diffeomorphisms have proved
frustrating.  We do not know if examples coming from algebraic maps and
time-one maps of flows, or from perturbations of these diffeomorphisms are
indicative of the general behaviour of partially hyperbolic systems.

One special successful case is for 3-manifolds with nilpotent fundamental
groups.  In this setting, the only manifolds supporting partially hyperbolic
diffeomorphisms are 3-dimensional nilmanifolds, that is, circle bundles over
$\bbT^2$ \cite{Parwani}.  These are the only manifolds where all partially
hyperbolic systems have been classified up to leaf conjugacy, a notion
analogous to topological equivalence for flows
\cite{ham-thesis}\cite{ham-nil}.

These manifolds are also the only known manifolds for which all partially
hyperbolic systems have stable and unstable foliations which are
quasi-isometric when lifted to the universal cover \cite{Parwani}.  While this
could be demonstrated as a consequence of the classification, such reasoning
would be circular.  In fact, the proof of classification starts by
establishing geometric properties of the foliation, including that of
quasi-isometry; using these properties to establish Global Product Structure;
then using the Global Product Structure to construct the leaf conjugacy.
The proof relies on results specific to the 3-dimensional manifolds \cite{BI}.
While there are more difficulties in higher dimensions, the same general
approach might work, in which case, Theorem \ref{introgps} would be a critical
step in the reasoning.

As a simple example of where Theorem \ref{introgps} applies, consider affine
maps of the form
\[
    \Rd \to \Rd, \quad x \mapsto A x + b
\]
where $A \in G L(n, \bbZ)$ and $b \in \Rd$.
If such a map has a partially hyperbolic splitting, the resulting leaves of
the invariant foliations are translates of subspaces of $\Rd$ and so the
foliations are quasi-isometric.

In \cite[Proposition 5]{Brin}, Brin considers a specific example of an
affine map which quotients down to a partially hyperbolic diffeomorphism
$f_0:\bbT^3 \to \bbT^3$.  He shows that every $f$ which is $\epsilon$-$C^1$-close
to $f_0$ is partially hyperbolic and has quasi-isometric invariant foliations.
Further, he gives an actual value for the size of this neighbourhood:
$\epsilon=0.1$.

This technique readily extends to any affine map $g_0:\bbT^d \to \bbT^d$.  Given
the eigenvalues associated to $g_0$, we may calculate a value for $\epsilon$
such that every $g$ $\epsilon$-$C^1$-close has quasi-isometric invariant
foliations.  Then, Theorem \ref{introgps} tells us that every system in this
macroscopic open set around $g_0$ has Global Product Structure.

Suppose further that $g_0$ is a skew-product, that is, its center leaves are
compact.  Take a subset $U \subset \bbT^d$ and a diffeomorphism $h:\bbT^d \to
\bbT^d$ whose support lies in $U$.  If $U$ is disjoint from the orbit of a
periodic center leaf of $g_0$, then $g := h \circ g_0$ will preserve this leaf.  
If $h$ is sufficently $C^1$-close to the identity, then, as above, $W^s$ and
$W^u$ are quasi-isometric, and by Theorem \ref{intronoholo}, every leaf is
compact and the center foliation is a trivial fiber bundle.  Even in cases
where $h$ is large in the $C^1$ norm, we may still be able to establish that
$W^s$ and $W^u$ are quasi-isometric (say, by \cite[Proposition 4]{Brin}).
Then, the same consequences follow.  These techniques apply, with minor
adaptations, to more general skew products on manifolds of the form $\bbT^d
\times N$ where $N$ has an abelian fundamental group.


\section{Splittings} \label{finer} 

We now state and prove several results for invariant splittings, from which
Theorem \ref{introgps} follows as a consequence.

Throughout this section, assume $M$ is a Riemannian manifold (not
necessary compact or connected, but without boundary) and that $f:M \to M$ is a
$C^1$-diffeomorphism.  Also assume that the derivative $Tf$ is uniformly
bounded on $M$, which is always the case when $M$ is compact or when $f$ is
lift of a map on a compact manifold to the universal cover.

\begin{defn}
    A $Tf$-invariant subbundle $E \subset TM$ is \emph{contracting} if there is
    $n>0$ such that
    $\|Tf^n v\| < \tfrac{1}{2}$
    for every unit vector $v \in E$.
    A $Tf$-invariant subbundle $E \subset TM$ is \emph{expanding} if there is
    $n>0$ such that
    $\|Tf^n v\|>2$
    for every unit vector $v \in E$.
\end{defn}
\begin{remark}
    Note that by this definition the zero bundle $0 \subset TM$ is both
    contracting and expanding.  It is the only subbundle to satisfy both
    properties.
\end{remark}
\begin{defn}
    Suppose $E^1,E^2$ are continuous $Tf$-invariant subbundles of $TM$.
    We say that $E^2$ \emph{dominates} $E^1$ if there is $n>0$ such that
    \[
        \|Tf^n u_x\| < \frac{1}{2} \|Tf^n v_x\|
    \]
    for all $x \in M$ and unit vectors $u_x \in E^1_x$, $v_x \in E^2_x$.
    Further, $E^2$ \emph{absolutely dominates} $E^1$ if there is $n>0$ such
    that
    \[
        \|Tf^n u_x\| < \frac{1}{2} \|Tf^n v_y\|
    \]
    for all $x,y \in M$ and unit vectors $u_x \in E^1_x$, $v_y \in E^2_y$
\end{defn}
\begin{remark}
    If $E^2$ dominates $E^1$, it is clear from the definition that
    $E^1 \cap E^2 = 0$.
\end{remark}
\begin{remark}
    Absolute domination is stronger just domination as it compares the effect
    of the derivative at different points.
    An equivalent definition of absolute domination is that there are
    constants $0<\gamma<\mu$ and an integer $n>0$ such that
    \[
        \|Tf^n u_x\| < \gamma < \mu < \|Tf^n v_x\|
    \]
    for all $x \in M$ and unit vectors $u_x \in E^1_x$, $v_x \in E^2_x$.
    From this, one can show that at least one of $E^1$ and $E^2$ is either
    expanding or contracting.
\end{remark}
\begin{defn}
    A \emph{splitting} of a continuous $Tf$-invariant subbundle $E \subset TM$
    is a sum of the form
    \[
        E = E^1 \oplus \cdots \oplus E^n
    \]
    where each $E^i$ is continuous and $Tf$-invariant.  A splitting is
    \emph{(absolutely) dominated} if the subbundles may be ordered such that
    $E^j$ (absolutely) dominates $E^i$ for $1  \le  i < j  \le  n$.
\end{defn}
\begin{remark}
    Extending the previous remark, one sees that in an absolutely dominated
    splitting, there can be at most one subbundle $E^i$ which is neither
    expanding nor contracting.
\end{remark}
\begin{remark}
    To make explicit the order of the subbundles in the dominated splitting,
    we sometimes write
    \[
        E = E^1 \oplus_< E^2 \oplus_< \cdots \oplus_< E^n
    \]
    where $E^i \oplus_< E^j$ signifies that $E^j$ dominates $E^i$.
\end{remark}
With this notation, we can now give succinct definitions for notions
introduced in the previous section. 

\begin{defn}
    The diffeomorphism $f:M \to M$ is \emph{Anosov} if there is a splitting
    $TM = \Es \oplus \Eu$ such that $\Es$ is contracting and $\Eu$ is expanding.
\end{defn}
\begin{defn}
    The diffeomorphism $f:M \to M$ is \emph{partially hyperbolic}
    if there is an absolutely dominated splitting
    $TM = \Es \oplus_< \Ec \oplus_< \Eu$ such that $\Es$ is contracting, $\Eu$ is
    expanding, and $0  \ne  \Ec  \ne  TM$.

    If both $\Es$ and $\Eu$ are non-zero then $f$ is partially hyperbolic in the
    \emph{strong sense}.  Otherwise, it is partially hyperbolic in the
    \emph{weak sense}.
\end{defn}
\begin{remark}
    This definition of partial hyperbolicity is sometimes
    called \emph{absolute} partial hyperbolicity, in contrast to
    \emph{point-wise} partial hyperbolicity, where the dominated splitting
    $TM = \Es \oplus_< \Ec \oplus_< \Eu$
    need not be an absolutely dominated splitting.
    In this more general setting, F.~Rodriguez-Hertz, M.~A.~Rodriguez-Hertz,
    and R.~Ures have announced a non-dynamically coherent example in $\T^3$.
    This example suggests that quasi-isometry is useful only in studying the
    absolute case.  Indeed, all of the proofs in this paper rely on the
    absolute version of the definition.
\end{remark}
\begin{defn}
    A subbundle $E \subset TM$ is \emph{integrable} if there is a foliation
    $\mathcal{F}$ such that $T \mathcal{F}=E$. It is \emph{uniquely integrable}
    if every curve $\gamma:[0,1] \to M$ tangent to $E$ lies in a single leaf of
    $\mathcal{F}$.
\end{defn}
\begin{thm}
    [Brin--Pesin  \cite{BrinPesin}, Hirsch--Pugh--Shub \cite{HPS}]
    \label{uintorig}
    If $f$ is partially hyperbolic, then $\Wu$ and $\Ws$ are
    uniquely integrable.
\end{thm}
To prove Theorem \ref{uintorig}, one constructs the foliations by a graph
transformation argument.  The same argument applies when restricted to a leaf
of an invariant foliation, and so can be generalized.
    
\begin{prop} \label{uint}
    If $f$ has a $Tf$-invariant, uniquely integrable subbundle
    $E \subset TM$ with dominated splitting $E = \Ec \oplus_< \Eu$ and $\Eu$ is
    expanding, then $\Eu$ is uniquely integrable.
\end{prop}
Let $d=d_M$ denote distance measured on the manifold $M$.

\begin{defn}
    A foliation $W$ on $M$ is \emph{quasi-isometric} if there is $Q>1$ such
    that $d_W(x,y)  \le  Q\,d(x,y)$ for any $x,y \in M$ which lie on the same
    leaf of $W$.
\end{defn}
The standard definition of quasi-isometry would use an inequality of the form
\[
    d_W(x,y)  \le  Q\,d(x,y) + Q.
\]
In this paper, we only consider foliations with $C^1$-leaves which are tangent
to a uniformly continuous subbundle $E \subset M$.  In this case, one can show
that as $d(x,y) \to 0$, the ratio between $d_W(x,y)$ and $d(x,y)$ tends
uniformly to one.  Therefore (up to a change of the value $Q$) the two
definitions are equivalent.

As stated in the introduction, quasi-isometry implies dynamical coherence.

\begin{thm}
    [Brin \cite{Brin}] \label{brinintorig}
    If $f$ is partially hyperbolic, and $\Wu$ and $\Ws$ are
    quasi-isometric when lifted to the universal cover, then
    $\Ec$, $\Ecs$, and $\Ecu$ are uniquely
    integrable.
\end{thm}
Note that Theorems \ref{uintorig} and \ref{brinintorig} hold for partial
hyperbolicity in the weak or strong sense.  If $\Es$ or $\Eu$ is equal to the
zero bundle, it is uniquely integrable and the tangent foliation, where each
leaf consists of a single point, is trivially quasi-isometric.
Theorem \ref{brinintorig} can be formulated with respect to an invariant
foliation.

\begin{prop} \label{brinint}
    If $f$ has a $Tf$-invariant, uniquely integrable subbundle
    $E \subset TM$ with absolutely dominated splitting $E = \Ec \oplus_< \Eu$ and
    $\Eu$ is expanding and tangent to a quasi-isometric foliation, then $\Ec$ is
    uniquely integrable.
\end{prop}
We now build up the knowledge needed to prove Theorem \ref{introgps}, starting
with a technical lemma.

\begin{lemma} \label{tribound}
    Suppose $f$ has a continuous, $Tf$-invariant subbundle $E \subset TM$ with
    absolutely dominated splitting $E = \Ec \oplus_< \Eu$
    and that $\Ec$ and $\Eu$ are tangent to quasi-isometric foliations $\Wc$ and
    $\Wu$.
    Then, there is $c>0$ such that
    \[
            \max\{d(x,y),d(x,z)\}  \le  c\,d(y,z)
    \]
    for all $x \in M$, $y \in \Wu(x)$, and $z \in \Wc(x)$.
\end{lemma}
\begin{remark}
    Since the foliations are quasi-isometric, we may, with at most a change in
    the constant $c$, conclude further that
    \[
        \max\{d_u(x,y),d_c(x,z)\}  \le  c\,d(y,z)
    \]
    where $d_u$ and $d_c$ are the distances measured along the foliations.
\end{remark}
\begin{remark}
    In this lemma, we need not assume that $E$ is integrable.
\end{remark}
\begin{proof}
    Without loss of generality, by replacing $f$ by an iterate $f^n$,
    there are constants $\gamma < \mu$ such that
    $\|Tf v\| \le \gamma \|v\|$ for $v \in \Ec$, and
    $\mu \|v\|  \le  \|Tf v\|$ for $v \in \Eu$.
    By standing assumption, the derivative $Tf$ is bounded; that is, there is
    $\lambda>0$ such that
    $\|Tf v\|  \le  λ \|v\|$ for all $v \in TM$.

    Let $x,y,z \in M$ be as in the hypothesis.

    \noindent {\bf Case One:} Suppose $2\,d(x,y)  \le  d(x,z)$.
    Then, using the triangle inequality,
    \[
        d(x,z)
        = 2d(x,z) - d(x,z)
         \le  2d(x,z) - 2d(x,y)
         \le  2d(y,z).
    \]
    Thus, the theorem holds in this case with the caveat that $c \ge 2$.

    \noindent {\bf Case Two:} Suppose $d(x,z) < 2\,d(x,y)$.
    Let $Q>1$ be a constant of quasi-isometry satisfied by both foliations.
    Then
    \[
        d(f^n(x),f^n(z))
         \le  d_c(f^n(x),f^n(z))
         \le  \gamma^n d_c(x,z)
         \le  Q \gamma^n d(x,z)
    \]
    and
    \[
        Q d(f^n(x),f^n(y))
         \ge   d_u(f^n(x), f^n(y))
         \ge   \mu^n d_u(x,y)
         \ge   \mu^n d(x,y).
    \]
    These estimates combine to give
    \begin{align*}
        d(f^n(x),f^n(y)) & \le  d(f^n(x),f^n(z)) + d(f^n(y),f^n(z))  \quad \Rightarrow \quad  \\
        Q \inv \mu^n d(x,y) & \le  Q \gamma^n d(x,z) + \lambda^n d(y,z)   \quad \Rightarrow \quad  \\
        Q \inv \mu^n d(x,y) & \le  2 Q \gamma^n d(x,y) + \lambda^n d(y,z)  \quad \Rightarrow \quad  \\
        (Q \inv \mu^n - 2 Q \gamma^n) d(x,y) & \le  \lambda^n d(y,z).
    \end{align*}
    Fixing $n$ large enough that $Q \inv \mu^n - 2 Q \gamma^n > 0$,
    the last line above simplifies to $d(x,y)  \le  c_0 d(y,z)$ for an
    appropriate positive constant $c_0$ independent of $x$, $y$, and $z$.  As
    we are considering the case $d(x,z) < 2\,d(x,y)$, it follows that
    \[
        \max\{d(x,y),d(x,z)\}  \le  2\,c_0\,d(y,z).
    \]
    Taking $c = \max\{2,2c_0\}$ establishes the desired inequality for
    Cases One and Two and concludes the proof.
\end{proof}
\begin{thm} \label{qigps}
    Suppose $f$ has a continuous, $Tf$-invariant subbundle $E \subset TM$ which
    is uniquely integrable with tangent foliation $W$.  Suppose further that
    $E$ has an absolutely dominated splitting
    $E = \Ec \oplus_< \Eu$
    and $\Wc$ and $\Wu$ are quasi-isometric foliations tangent to $\Ec$ and $\Eu$.
    Then, for any two points $x$ and $y$ on a leaf of $W$, the intersection
    $\Wc(x) \cap \Wu(y)$ consists of exactly one point.
\end{thm}
\begin{proof}
    First, consider the case where $E=TM$ and the leaves of $W$ are the
    connected components of $M$.

    To prove uniqueness, assume $x$ and $y$ lie on the same $\Wc$ leaf and the
    same $\Wu$ leaf.  Lemma \ref{tribound} with $y=z$ states that
    \[
        \max \{ d(x,y), d(x,y) \}  \le  c\,d(y,y) = 0,
    \]
    which is a convoluted way of saying $x=y$.

    To establish existence of the intersection, first fix a leaf $L$ of $\Wc$.
    Define, for $T>0$,
    \[    
        A_T = \bigcup_{x \in L} A^u_T(x) \quad \text{where} \quad
        A_T^u(x) = \{ y \in \Wu(x) : d_u(x,y)  \le  T \}
    \]
    and
    \[
        B_T = \bigcup_{x \in L} B^u_T(x) \quad \text{where} \quad
        B_T^u(x) = \{ y \in \Wu(x) : d_u(x,y) = T \}.
    \]
    As the leaf $L$ is transverse to $\Wu$, one can show by continuity of the
    foliation that the $d_u$-distance from a point $y \in A_T$ to $L$ depends
    continuously on $y$.  Therefore,
    if $y \in A^u_T(x)$ and $d_u(x,y)<T$ then $y$ is in the
    interior of $A_T$, and consequently $\partial A_T \subset B_T$.

    For some value $T>0$, suppose $x \in L$ and $y \in B^u_T(x)$.  Note that the
    leaf $L$ is properly embedded in $M$, for if it accumulated on a point
    $p$, then $\Wu(p)$ would intersect $L$ more than once.
    Then, the intersection of $L$ with a large closed ball centered at $y$ is
    a compact set.  The function $z \mapsto d(y,z)$ defined on this compact set
    achieves its minimum at a point $z \in L$ such that
    $\dist(y, L) = d(y, z)$.  The points $x,y,z$ satisfy the conditions of
    Lemma \ref{tribound} and there is a constant $c>0$ (independent of $T$)
    such that $d_u(x,y)  \le  c\,d(y,z)$.  Then,
    \[
        \dist(y, L) = d(y,z)  \ge  \frac{1}{c} d_u(x,y) = \frac{T}{c}.
    \]
    As this holds for all $y \in \partial A_T \subset B_T$, it follows that
    $\dist(\partial A_T, L)  \ge  \tfrac{T}{C}$.
    Any $y$ in the same component as $x$ is a finite distance away from $L$.
    Then, $y$ is in $A_T$ for sufficiently large $T$ and hence $\Wu(y)$
    intersects $L$.

    In the case where $E  \ne  TM$, consider the
    manifold $\Lspace$ defined as the disjoint union of the leaves of the
    foliation $W$.  There is
    a canonical bijective map $i:\Lspace \to M$, so we may define a Riemannian
    metric on $\Lspace$ by declaring $i$ on each leaf to be an isometry to its
    image.
    Further, there is a unique diffeomorphism $g:\Lspace \to \Lspace$ such that
    $i \circ g=f \circ i$.  The proof reduces to the previous case by
    considering $g$, $\Lspace$, and $E = T \Lspace$ in place of $f$, $M$, and
    $E = TM$.
\end{proof}
\begin{thm} \label{nsum}
    Suppose $f$ has a continuous, $Tf$-invariant, uniquely integrable
    subbundle $E \subset TM$ with absolutely dominated splitting
    \[
        E = E^1 \oplus_< E^2 \oplus_< \cdots \oplus_< E^n
    \]
    such that each $E^i$ is tangent to a quasi-isometric foliation.
    Then each subbundle of the form 
    $E^j \oplus \cdots \oplus E^k$
    for $1  \le  j  \le  k  \le  n$
    is uniquely integrable and the tangent foliation is quasi-isometric.
\end{thm}
To prove this theorem, first consider the case $n=2$.

\begin{lemma} \label{twosum}
    Suppose $f$ has a continuous, $Tf$-invariant, uniquely integrable
    subbundle $E \subset TM$ with absolutely dominated splitting
    $E = \Ec \oplus_< \Eu$
    such that $\Ec$ and $\Eu$ are tangent to quasi-isometric foliations $\Wc$
    and $\Wu$.  Then, the foliation tangent to $E$ is quasi-isometric.
\end{lemma}
\begin{proof}
    Suppose $\Phi(x,y,z)$ and $\Psi(x,y,z)$ are real-valued formulae which are
    well-defined when $x \in M$, $y \in \Wu(x)$, and $z \in \Wc(x)$.  Adopt the
    notation $\Phi(x,y,z) \myrel \Psi(x,y,z)$ if there is a constant $K>1$ such
    that
    \[
        \frac{1}{K} \Psi(x,y,z)  \le  \Phi(x,y,z)  \le   K \Psi(x,y,z)
    \]
    for all such triples of points.

    For instance, with this notation, Lemma \ref{tribound} can be concisely
    stated as
    \begin{equation} \label{maxM}
        \max \{ d_M(x,y), d_M(x,z) \} \myrel d_M(y,z)
    \end{equation}
    where $d=d_M$ is the metric on $M$.  Here, we have applied the lemma to
    the diffeomorphism $f$ on the manifold $M$.
    As in the proof of Theorem \ref{qigps}, let $\Lspace$ be the disjoint union
    of the leaves of $W$ with $i:\Lspace \to M$ the canonical bijection and
    $g:\Lspace \to \Lspace$ defined by the relation $i \circ g = f \circ i$.
    Applying Lemma \ref{tribound} to $g$ implies that
    \begin{equation} \label{maxW}
        \max \{ d_W(x,y), d_W(x,z) \} \myrel d_W(y,z)
    \end{equation}
    where $d_W$ denotes the distance along a leaf of $W$.  

    Using the quasi-isometry of $\Wu$ in $M$,
    \[
        d_M(x,y)  \le  d_W(x,y)  \le  d_u(x,y)  \le  Q d_M(x,y)
    \]
    so that $d_W(x,y) \myrel d_M(x,y)$ and similarly $d_W(x,z) \myrel d_M(x,z)$.
    These two relations combine to give
    \begin{equation} \label{maxmax}
        \max \{ d_W(x,y), d_W(x,z) \} \myrel \max \{ d_M(x,y), d_M(x,z) \}.
    \end{equation}
    It is clear that the relation ``$\myrel$'' is transitive, so that the
    relations \eqref{maxM}, \eqref{maxW}, and \eqref{maxmax} imply
    \begin{equation} \label{Wqi}
        d_W(y,z) \myrel d_M(y,z).
    \end{equation}
    Now, if $y$ and $z$ are any two points on the same leaf of $W$, Theorem
    \ref{qigps} implies that there is $x \in \Wu(y) \cap \Wc(z)$.  That is,
    $(x,y,z)$ is a valid triple and the above relation \eqref{Wqi} states that
    the distances $d_W(y,z)$ and $d_M(y,z)$ differ by at most a bounded
    proportion; in other words, the foliation $W$ is quasi-isometric.
\end{proof}
\begin{proof}
    [Proof of Theorem \ref{nsum}.]
    Assume that the theorem has already been proven for all splittings with
    less than $n$ summands and consider a splitting with exactly $n$ summands.
    Define $E(i,j)=E^i \oplus \cdots \oplus E^j$, so that
    $E=E(1,n)$ is the full bundle.  For each $i$, $E = E(1,i) \oplus_< E(i+1,n)$
    is an absolutely dominated splitting and either $E(1,i)$ is contracting or
    $E(i+1,n)$ is expanding.  Assume the latter holds.  Then, Proposition
    \ref{uint} implies that $E(i+1,n)$ is uniquely integrable.  By the
    inductive hypothesis, all of the subbundles $E(j,k)$ for $i+1 \le j \le k \le n$
    are tangent to quasi-isometric foliations.  In particular, the foliation
    tangent to $E(i+1,n)$ is quasi-isometric and Proposition \ref{brinint}
    implies $E(1,i)$ is uniquely integrable.

    We have shown that every bundle of the form $E(1,i)$ or $E(j,n)$ is
    uniquely integrable.  As the property of unique integrability survives
    under intersection, each $E(i,j) = E(1,j) \cap E(i,n)$ is uniquely
    integrable.  The resulting foliation tangent to $E(i,j)$ is
    quasi-isometric when $j-i<n-1$ due to the inductive hypothesis.

    To see that the foliation tangent to $E=E(1,n)$ is quasi-isometric, write
    it as $E(1,1) \oplus E(2,n)$ and apply Lemma \ref{twosum}.
\end{proof}
Theorem \ref{introgps} now follows as a consequence of Theorems \ref{qigps} and
\ref{nsum}.

\section{Holonomy} \label{secholo} 
\begin{defn}
    Let $\F$ be a foliation on a manifold $M$, $\alpha$ a closed curve on a
    leaf of $\F$, and $\tau \subset M$ a small plaque transverse to $\F$
    passing through a point $x$ on $\alpha$ and such that $\dim \F + \dim \tau = \dim
    M$.  Then, there is a small neighbourhood $\tau_0 \subset \tau$ containing
    $x$ and a map $h:\tau_0 \to \tau$ defined by following leaves of $\F$ along
    paths close to $\alpha$.  The germ of this map is the holonomy along
    $\alpha$.  If, for a foliation $\F$, every such choice of $\alpha$, $\tau$,
    and $\tau_0$ yields a map $h:\tau_0 \to \tau$ which is the identity map, we
    say that $\F$ is \emph{without holonomy}.
\end{defn}
\begin{prop} \label{noholo}
    Suppose $M$ is a compact connected manifold with an abelian fundamental
    group, and $f:M \to M$ is partially hyperbolic.  If $\Wu$ and $\Ws$
    are quasi-isometric when lifted to the universal cover, then $\Wc$ is
    without holonomy.
      \end{prop}
\begin{remark}
    This is a restatement of Theorem \ref{intronoholo}.  The two are equivalent
    as the center-stable bundle of a diffeomorphism which is partially
    hyperbolic in the strong sense can always be viewed as the center bundle
    of the same system regarded as partially hyperbolic in the weak sense.
\end{remark}
\begin{proof}
    First, consider the weak sense of partial hyperbolicity where $\Es=0$.

    Suppose holonomy along the center foliation is non-trivial along a closed
    path $\alpha: [0,1] \to \Wc(x)$ based at a point $x \in M$.  We are free to
    consider the holonomy as defined on a small plaque of $\Wu(x)$.  Then, as
    depicted in Figure \ref{holofig}, there are
    are distinct points $y,z \in \Wu(x)$ and paths $\beta: [0,1] \to \Wc(y)$,
    $\gamma, \theta: [0,1] \to \Wu(x)$ such that
    \begin{align*}
        \alpha(0) &= x & \alpha(1) &= x \\
        \beta(0) &= y & \beta(1) &= z \\
        \gamma(0) &= z & \gamma(1) &= y \\
        \theta(0) &= x & \theta(1) &= y
    \end{align*}
    and the concatenation $\theta \cdot \beta \cdot \gamma \cdot \theta \inv$ is
    homotopic to $\alpha$.
    \begin{figure}
        \begin{center}
            \includegraphics{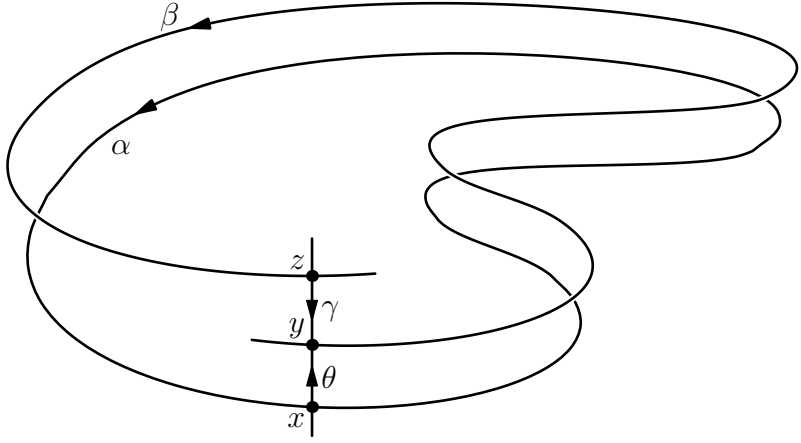}
            \caption{A non-trivial holonomy.}
            \label{holofig}
        \end{center}  \end{figure}
    For $n \ge 0$, define
    \[
        x_n = f^n(x), \quad
        y_n = f^n(y), \quad
        \alpha_n = f^n \circ \alpha, \quad
        \beta_n = f^n \circ \beta, \quad
        \gamma_n = f^n \circ \gamma.
    \]
    Define paths
    $\phi_n:[0,1] \to M$, such that $\phi_n(0) = x_n$, $\phi_n(1) = y_n$ and
    \[    
        \length \phi_n  \le  \diam M.
          \]
    Note that, unlike the above sequences, we do not require $\phi_n$ to be
    equal to $f^n \circ \phi_0$.  Also note that
    \[    
        [\alpha_n] =
        [\theta_n \cdot \beta_n \cdot \gamma_n \cdot \theta_n \inv] =
        [\phi_n \cdot \beta_n \cdot \gamma_n \cdot \phi_n \inv] \in \pi_1(M,x_n)
    \]
    where last equality follows by conjugating by $[\phi_n \cdot \theta_n \inv]$
    and using the assumption that the fundamental group is abelian.

    Now consider the universal cover $\tilde M$.  Choose $\tilde x_n \in \tilde M$
    over each point $x_n$ and lift each curve $\alpha_n$ to a curve $\tilde
    \alpha_n$ whose starting point $\tilde \alpha_n(0)$ is $\tilde x_n$.

    Lift $\phi_n$, $\beta_n$, $\gamma_n$, and again $\phi_n$ to curves $\tilde
    \phi_n$, $\tilde \beta_n$, $\tilde \gamma_n$, and $\ttphi_n$ so that
    the concatenation
    $\tilde \omega_n :=
    \tilde \phi_n \cdot \tilde \beta_n \cdot \tilde \gamma_n \cdot \ttphi_n \inv$
    is a continuous curve starting at $\tilde x_n$.
    Since
    $\tilde \alpha_n(0) = \tilde \omega_n(0)$
    and the projections of $\tilde \alpha_n$ and $\tilde \omega_n$ to $M$ are
    homotopic,
    it follows that 
    $\tilde \alpha_n(1) = \tilde \omega_n(1)$.

    There are constants $1 < \lambda < \mu$ such that
    \[
        \|Tf^n v\| < \lambda^n < \mu^n < \|Tf^n w\|
    \]
    for unit vectors $v \in \Ec$, $w \in \Eu$ and sufficiently large $n$.
    Then
    \[
        d(\tilde \alpha_n(0), \tilde \alpha_n(1))
         \le  \length \tilde \alpha_n
         \le  \lambda^n \length \alpha
    \]
    and, if $Q$ is the constant of quasi-isometry for $\Wu$,
    \begin{align*}
        d(\tilde \omega_n(0), \tilde \omega_n(1))
        & \ge  d(\tilde \gamma^n(0), \tilde \gamma^n(1))
        - \length \tilde \beta^n
        - \length \tilde \phi^n
        - \length \ttphi^n \\
        & \ge  Q \inv \mu^n d_u(\gamma(0), \gamma(1))
        - \lambda^n \length \beta
        - 2 \diam M.
    \end{align*}
    For large $n$, these two estimates yield a contradiction.  This finishes
    the case when $\Es=0$.  The case where $\Eu=0$ follows by analogy.

    In the case that both $\Eu$ and $\Es$ are non-zero, first observe
    that the foliations $\Wcs$ tangent to $\Es \oplus \Ec$ and $\Wcu$ tangent to
    $\Ec \oplus \Eu$ are without holonomy, as each can be regarded as center
    foliation of a system which is partially hyperbolic in the weak sense.
    Then, $\Wc$ is without holonomy as it is the intersection of two transverse
    foliations without holonomy.
\end{proof}
We now prove Theorem \ref{introhomeo}.  That is, assuming Global Product
Structure and that $\Wc$ is without holonomy, we show that all of the center
leaves are homeomorphic.

\begin{proof}
    [Proof of Theorem \ref{introhomeo}.]
    Let $p:\tilde M \to M$ be the universal covering map and fix $x_0,y_0 \in
    \tilde M$.  Our goal is to construct a homeomorphism $h:\Wc(x_0) \to
    \Wc(y_0)$ such that for $x_1, x_2 \in \Wc(x)$, $p(x_1) = p(x_2)$ if and only
    if $p(h(x_1)) = p(h(x_2))$.  Then, $h$ will descend to a homeomorphism of
    the leaves $\Wc(p(x_0))$ and $\Wc(p(y_0))$ on $M$.

    First, consider the case where $x_0$ and $y_0$ lie on the same
    center-unstable leaf (which is the only case if $\Es=0$).
    Define $h$ as the map which assigns $x \in \Wc(x_0)$ to
    the unique intersection of $\Wu(x)$ and $\Wc(y_0)$.  This map exists and is
    continuous due to the Global Product Structure.

    Suppose $x_1, x_2 \in \Wc(x_0)$ satisfy $p(x_1)=p(x_2)$.  Take paths
    $\alpha:[0,1] \to \Wc(x)$ and $\beta:[0,1] \to \Wu(x)$ such that
    $\alpha(0)=\beta(0)=x_1$, $\alpha(1)=x_2$ and $\beta(1)=h(x_1)$.

    Define $\phi:[0,1] \times [0,1] \to \tilde M$ by
    $\phi(s,t) = \Wu(\alpha(s)) \cap \Wc(\beta(t))$.
    Again, Global Product Structure guarantees that this is well-defined and
    continuous.  Consider the set
    \[
        S = \{ t \in [0,1] : p(\phi(0,t)) = p(\phi(1,t)) \}.
    \]
    It is closed due to the continuity of $p$ and $\phi$.  The assumption that
    the center foliation is without holonomy implies that $S$ is open
    in the relative topology of $[0,1]$, and as $0 \in S$ by the definition of
    $\phi$, it follows that $S = [0,1]$.  In particular,
    \[
        p(h(x_1)) = p(\phi(0,1)) = p(\phi(1,1)) = p(h(x_2)).
    \]
    By reversing the roles of $x_0$ and $y_0$, it is easy to find an inverse
    for $h$ and show that
    $p(h(x_1)) = p(h(x_2))$ implies $p(x_1) = p(x_2)$.  Thus, $h$ descends to a
    homeomorphism of $\Wc(p(x_0))$ and $\Wc(p(y_0))$ as center leaves on the
    manifold $M$.

    That finishes the case of two center leaves on the same center-unstable
    leaf.  The same argument applies to center leaves on the same
    center-stable leaf, and then, by composing such homeomorphisms, one can
    construct a homeomorphism between any two center leaves on $M$.
\end{proof}
\medskip

The last task of this section is to prove Theorem \ref{easynil}.
We use the following classification given by Parwani
\cite[Theorem 1.10]{Parwani}.

\begin{thm}
    [Parwani, \cite{Parwani}] \label{solvclass}
    Suppose $f:M \to M$ is partially hyperbolic
    with $\dim \Es = \dim \Ec = \dim \Eu = 1$.
    \begin{enumerate}
        \item If $\pi_1(M)$ is solvable, $M$ is finitely covered by a torus
        bundle over the circle.
        \item If $\pi_1(M)$ is nilpotent, $M$ is a circle bundle over the torus.
      \end{enumerate}  \end{thm}
The first step in proving Theorem \ref{easynil} is to classify the possible
leaves of the foliations.

\begin{prop} \label{cyclicleaf}
    Suppose $f:M \to M$ is partially hyperbolic with Global Product
    Structure and $\dim \Es = \dim \Ec = \dim \Eu = 1$.  Then, every leaf of $\Wcs$
    is either a plane, a cylinder, or a M\"obius band.
\end{prop}
\begin{remark}
    This is a statement about leaves on $M$, not the universal cover $\tilde M$
    where the leaves are all planes.
\end{remark}
\begin{remark}
    In the proof, we use the known fact that each leaf of an $n$-dimensional
    stable or unstable foliation is homeomorphic to $\R^n$.
\end{remark}
\begin{proof}
    Fix $S \subset M$ a $cs$-leaf and $L$ a center leaf lying in $S$.  The
    foliation $\Wcs$ is Reebless \cite{BBI1}.  Therefore, the embedding $S
    \hookrightarrow M$ induces an injection $\pi_1(S) \hookrightarrow \pi_1(M)$
    \cite{cancon2}\cite{solodov1984components}.
    Let $\tilde L \subset \tilde M$ be a connected component of the pre-image of
    $L$ under the universal covering map $p:\tilde M \to M$.  Then
    $\pi_1$-injectivity implies $\tilde L$ is simply connected and therefore a
    line.  For distinct points $x,y \in \tilde L$, the stable manifolds $\Ws(x)$
    and $\Ws(y)$ are disjoint subsets of $\tilde M$ and therefore $\tilde S :=
    \bigcap_{x \in \tilde L} \Ws(x),$ a line bundle over a line, is topologically
    a plane.  It follows from Global Product Structure that $\tilde S$ is a
    connected component of $p \inv(S)$ and that $p(\tilde S)=S$.
    
    The set $\tilde S$ can be viewed as the universal cover of $S$ and each
    $\alpha \in \pi_1(S)$ can be viewed as a deck transformation $\alpha:\tilde S
    \to \tilde S$.  Define a group action of $\pi_1(S)$ on $\tilde L$ by defining
    $\alpha \cdot x$ as the unique intersection of $\Ws(\alpha(x))$ and $L$.  If
    $\alpha \cdot x = x$, then $\alpha(x)$ and $x$ would lie on the same stable
    leaf and, assuming $\alpha$ is non-trivial, projection down to $M$ would
    yield a stable leaf which is a closed circle, a contradiction.  Therefore,
    the action of $\pi_1(S)$ on the line $L$ is free.  A classic theorem of
    H\"older concerning such actions implies that $\pi_1(S)$ is abelian (for
    instance, see \cite[Theorem 6.10]{ghys2001groups}).

    If the subbundle $\Eu$ is not orientable, there is a double cover of $M$
    for which the lift of $\Eu$ is orientable.  Further, there is a unique
    lift of $f$ to a diffeomorphism of the cover which preserves this
    orientation; this map is also partially hyperbolic.  Thus, modulo a
    finite cover, we may assume $\Eu$ is orientable.  By the same reasoning,
    assume $\Ecs$ is orientable as well.
    Then the $cs$-leaf $S$ is orientable with an abelian fundamental group and
    is either a plane, a cylinder, or a torus.
    
    Assume for the purpose of
    contradiction that $S$ is a torus.  By Global Product Structure, every
    unstable leaf in $\tilde M$ intersects $\tilde S$ and therefore every
    unstable leaf in $M$ intersects $S$.  The orientation of $\Eu$ defines a
    flow along unstable leaves, and $S$ can be viewed as a global cross
    section.  This defines a first return map $\phi:S \to S$.
    If the induced map $\phi_* : H_1(S) \to H_1(S)$ on the first homology group
    is hyperbolic, the Lefschetz fixed point theorem implies that the flow
    along unstable leaves has a periodic orbit, but this is impossible.
    If $\phi_*$ is not hyperbolic, one can show that $\pi_1(M)$ is nilpotent.
    Then by Theorem \ref{solvclass}, $M$ is a nilmanifold, and the classification
    results in \cite{ham-thesis} and \cite{ham-nil} imply that no $cs$-leaf on
    such a manifold can be a torus.

    We have shown that, when lifted to a finite cover, every leaf of $\Wcs$ is
    either a plane or a cylinder, from which the result follows.
\end{proof}
\begin{proof}
    [Proof of Theorem \ref{easynil}.]
    It is a property of codimension-one foliations without holonomy that if
    $S$ is a leaf, then $\pi_1(S) \hookrightarrow \pi_1(M)$ is injective and,
    after identifying $\pi_1(S)$ with its image, that
    $[\pi_1(M), \pi_1(M)] \subset \pi_1(S).$  See \cite{li2002commutator} for a
    proof in the $C^0$ setting.

    By Proposition \ref{cyclicleaf}, $\pi_1(S)$ is cyclic for any leaf $S$
    of $\Wcs$.  Therefore, the commutator subgroup
    $[\pi_1(M), \pi_1(M)]$ is cyclic, and the fundamental group is
    solvable.  Theorem \ref{solvclass} implies, in particular, that $\pi_1(M)$ is
    infinite.  A classification of all possible solvable fundamental groups
    arising from 3-manifolds is given in \cite{evans1972solvable}.  If such a
    group is infinite and has cyclic commutator subgroup, it is nilpotent.
    Then, Theorem \ref{solvclass} completes the proof.
      \end{proof}
%
%

\section{Examples} \label{anosovexamples} 
The Heisenberg group $\Heis$ consisting of matrices of the form
\[
    \begin{pmatrix}
    1 & x & z \\ 0 & 1 & y \\ 0 & 0 & 1
    \end{pmatrix}
\]
is a Lie group.  The corresponding Lie algebra $\h$ is generated by the
elements
\[
    X = \begin{pmatrix}
    0 & 1 & 0 \\ 0 & 0 & 0 \\ 0 & 0 & 0
    \end{pmatrix},\quad
    Y = \begin{pmatrix}
    0 & 0 & 0 \\ 0 & 0 & 1 \\ 0 & 0 & 0
    \end{pmatrix},\quad \text{and}\quad
    Z = \begin{pmatrix}
    0 & 0 & 1 \\ 0 & 0 & 0 \\ 0 & 0 & 0
    \end{pmatrix}
\]
with $[X,Y]=Z$ and all other brackets zero.  The Lie algebra can also be
viewed as the set of those vector fields which are invariant under the
derivative $d L_g : T\Heis \to T\Heis$ of the action of left multiplication of
the group $L_g : \Heis \to \Heis, h \mapsto g h$.  As vector fields, $X$, $Y$,
and $Z$ integrate to one-dimensional foliations on $\Heis$.

Define a Riemannian metric on $\Heis$ by declaring $X$, $Y$, and $Z$ to form an
orthonormal basis at each point $p \in \Heis$.  One can verify that the
foliations tangent to $X$ and $Y$ are quasi-isometric, but the foliation
tangent to $Z$ is not.  Viewing a two-dimension subspace of $\h$ as a
subbundle of $T\Heis$, $X \oplus Z$ and $Y \oplus Z$ are each integrable and the
resulting two-dimensional foliations are not quasi-isometric.

In \cite{smale}, S.~Smale describes a six-dimensional nilmanifold supporting
Anosov diffeomorphisms.  It is covered by the Lie group $\Heis \times \Heis$
consisting of two copies of the Heisenberg group.  The corresponding Lie
algebra $\h \times \h$ is generated by the elements $X_1, Y_1, Z_1, X_2, Y_2,
Z_2$ where $[X_1,Y_1]=Z_1$, $[X_2,Y_2]=Z_2$, and all other brackets are zero.
There is a lattice $\Gamma \subset \Heis \times \Heis$ defining a compact nilmanifold
$N := (\Heis \times \Heis) / \Gamma$ and a constant $\lambda>1$ associated to
$\Gamma$ such that for any $a,b \in \Z \setminus \{0\}$, the Lie algebra
automorphism
\begin{align*}
    X_1 &\mapsto \lambda^a X_1 & X_2 &\mapsto \lambda^{-a} X_2 \\
    Y_1 &\mapsto \lambda^b Y_1 & Y_2 &\mapsto \lambda^{-b} Y_2 \\
    Z_1 &\mapsto \lambda^{a+b} Z_1 & Z_2 &\mapsto \lambda^{-a-b} Z_2 \\
\end{align*}
induces a Lie group automorphism which takes $\Gamma$ to itself.  Therefore,
each $a,b \in \Z \setminus \{0\}$ defines a diffeomorphism $f_{a,b}:N \to N$ on the
nilmanifold.  If $a+b  \ne  0$, one can verify that $f_{a,b}$ is Anosov.

In particular, if $a=1$ and $b=2$, then $f_{a,b}$ is Anosov with the splitting
\[
    \Eu = \langle X_1, Y_1, Z_1 \rangle \quad \text{and} \quad
    \Es = \langle X_2, Y_2, Z_2 \rangle
\]
and the resulting stable and unstable foliations are quasi-isometric.
If, instead, $a=2$ and $b=-5$, then the splitting is 
\[
    \Eu = \langle X_1, Y_2, Z_2 \rangle \quad \text{and} \quad
    \Es = \langle X_2, Y_1, Z_1 \rangle
\]
and neither foliation is quasi-isometric.

In either of these two examples, the diffeomorphism can alternatively be viewed
as partially hyperbolic by grouping $X_1$ and $X_2$ into a center bundle.
Thus, a dynamically coherent partially hyperbolic system on a nilmanifold may
or may not have quasi-isometric stable and unstable foliations.

\bigskip

\acknowledgement
The author would like to thank Enrique Pujals, Ali Tahzibi, 
and Amie Wilkinson for helpful discussions, as well as the anonymous
reviewer for suggestions to improve the paper's presentation.
The author gratefully acknowledges the financial support of
CNPq (Brazil).


\medskip 

{\sc\small
IMPA,
Estrada Dona Castorina 110,
Rio de Janeiro / Brasil\ \  22460-320 }

\bibliographystyle{plain}
\bibliography{dynamics}

\end{document}